\documentclass[11pt,reqno]{amsart} 

\usepackage[utf8]{inputenc}
\usepackage[a4paper]{geometry}

\usepackage{amsmath}
\usepackage{amsfonts}
\usepackage{amssymb}
\usepackage{amsthm}
\usepackage{paralist}

\usepackage{natbib}
\usepackage{color}
\usepackage[colorlinks=true,linkcolor=blue,citecolor=blue,pdfborder={0 0 0}]{hyperref}

\usepackage{dsfont}
\usepackage{bm}

\usepackage{graphicx}
\usepackage{enumerate}

\newtheorem{theorem}{Theorem}[section]
\newtheorem{proposition}[theorem]{Proposition}
\newtheorem{lemma}[theorem]{Lemma}

\theoremstyle{definition}
\newtheorem{definition}[theorem]{Definition}

\theoremstyle{remark}

\numberwithin{equation}{section}

\newcommand{\R}{\mathbb{R}}

\newcommand{\N}{\mathbb{N}}
\newcommand\Prob{\mathbb{P}}    

\newcommand{\scs}{\scriptscriptstyle}

\newcommand{\eps}{\varepsilon}

\newcommand{\ind}{\operatorname{\bf{1}}}
\newcommand{\diff}{{\,\mathrm{d}}}

\newcommand{\floor}[1]{\lfloor{#1}\rfloor}

\newcommand{\BM}{{b}}
\newcommand{\POT}{{p}}
\newcommand{\M}{{m}}

\newcommand{\RV}{{\operatorname{RV}}}

\newcommand{\bu}{\bm{u}}

\newcommand{\Exp}{\operatorname{E}}

\newcommand{\MSE}{\operatorname{MSE}}

\allowdisplaybreaks[4]

\begin{document}

\title[Multivariate Second Order Conditions]{On Second Order Conditions in the Multivariate Block Maxima and Peak over Threshold Method}

\author{Axel B\"ucher}
\address{Heinrich-Heine-Universität D\"usseldorf,
Mathematisches Institut,
Universit\"atsstr.~1, 40225 D\"usseldorf, Germany.}
\email{axel.buecher@hhu.de}

\author{Stanislav Volgushev}
\address{Department of Statistical Sciences, 
University of Toronto, 100 St. George St. 
Toronto, M5S 3G3 
Ontario, Canada}
\email{stanislav.volgushev@utoronto.ca}

\author{Nan Zou}
\address{Department of Statistical Sciences, 
University of Toronto, 100 St. George St. 
Toronto, M5S 3G3 
Ontario, Canada}
\email{nan.zou@utoronto.ca}

\date{\today}

\begin{abstract}
Second order conditions provide a natural framework for establishing
asymptotic results about estimators for tail related quantities. Such conditions are typically tailored to the estimation principle at hand, and may be vastly different for
estimators based on the block maxima (BM) method or the peak-over-threshold (POT) approach.
In this paper we provide details on the relationship between typical second
order conditions for BM and POT methods in the multivariate case.
We show that the two conditions typically imply each other, but
with a possibly different second order parameter. The latter implies that,
depending on the data generating process, one of the two methods can attain faster convergence rates than the other. The class of multivariate Archimax copulas is examined in detail;
we find that this class contains models for which the second order parameter is smaller for the BM method and vice versa. The theory is illustrated by a small simulation study. 
\medskip

\end{abstract}

\maketitle

\section{Introduction}

Extreme value theory is concerned with describing the tail behavior of a possibly multivariate distribution. Respective statistical models and methods find important applications in fields like finance, insurance, environmental sciences, hydrology or meteorology. 
In the multivariate case, a key part of statistical inference is estimation of the dependence structure. Mathematically, the dependence structure can be described in various equivalent ways (see, e.g., \citealp{Res87, BeiGoeSegTeu04, DehFer06}): by the stable tail dependence function $L$ \citep{Hua92}, by the exponent measure $\mu$ \citep{BalRes77},  by the Pickands dependence function $A$ \citep{Pic81}, by the tail copula $\Lambda$ \citep{SchSta06}, by the spectral measure $\Phi$ \citep{DehRes77}, by the madogram $\nu$ \citep{NavGuiCooDie09}, or by other less popular objects. 

Estimators for these objects typically rely on one of two basic principles allowing one to move into the tail of the distribution: the \textit{block maxima method} (BM) and the \textit{peak-over-threshold approach} (POT). More precisely, suppose that $\bm X_1, \dots, \bm X_n$, with $\bm X_i=(X_{i,1}, \dots, X_{i,d})'$, is an i.i.d.\ sample from a multivariate cumulative distribution function $F$. For some large number $k$ (in the asymptotics, one commonly considers $k=k_n\to\infty$ such that $k=o(n)$), let 
\[
\mathcal X_\POT = \{ \bm X_i \mid
\text{rank}(X_{i,j} \text{ among } X_{1,j}, \dots, X_{n,j}) \ge n-k  \text{ for some } j=1, \dots, d\},
\]
that is, $\mathcal X_\POT$ comprises all observations for which at least one coordinate is large. Any estimator defined in terms of these observations represents the multivariate POT method. The vanilla nonparametric estimator within this class is probably the \textit{empirical stable tail dependence function} \citep{Hua92}. 

To introduce the BM approach, let $1\le r \le n$ denote a large block size, and let $k=\lfloor n/r \rfloor$ denote the number of blocks (again, in the asymptotics, one commonly considers $k=k_n\to\infty$ such that $k=o(n)$). For $\ell=1, \dots, k$, let $\bm M_{\ell,r} =(M_{\ell, 1,r}, \dots, M_{\ell,1,r})'$ denote the vector of componentwise block-maxima in the $\ell$th block of observations of size $r$, that is, $M_{\ell, j,r}=\max(X_{i,j}: (\ell-1)r+1 \le i \le  \ell r\} )$. Any estimator defined in terms of the sample 
\[
\mathcal X_{\BM} = (\bm M_{1,r}, \dots, \bm M_{k,r})
\] 
represents the BM approach. 

Asymptotic theory for estimators based on the POT approach is typically formulated under a suitable second order condition (see Section~\ref{sec:secor} below for details). The asymptotic variance of resulting estimators is then typically of the order $k_n^{\scs -1}$ (see, e.g., \citealp{Hua92, EinKraSeg12, EinSeg09, SchSta06,FouDehMer15}, among others), whereas the rate of the bias is given by $(n/k_n)^{\rho_\POT}$, with $\rho_\POT <0$ denoting the \textit{second order parameter} in the aforementioned second order condition. Balancing the bias and variance leads to the choice $k_n\asymp n^{-2\rho_\POT/(1-2\rho_\POT)}$, which results in an asymptotic MSE of order $n^{2\rho_\POT/(1-2\rho_\POT)}$. For a particular class of models, this resulting convergence rate is in fact minimax-optimal \citep{DreHua98}.

Perhaps surprisingly, results on asymptotic theory for estimators based on the BM approach are typically based on the assumption that the block size $r$ is fixed and that the sample $\mathcal X_{\BM}$ is a genuine i.i.d. sample from the limiting attractor distribution (see, e.g., \citealp{GenSeg09} and references therein). Thereby, a potential bias is completely ignored and a fair comparison between estimators based on the POT and the BM approach is not feasible. This imbalance has recently been recognized by \cite{Dom15,FerDeh15,BucSeg18,DomFer17} in the univariate case; see also the overview article  \cite{BucZho18}. To the best of our knowledge, the only reference in the multivariate case is \cite{BucSeg14}. {In analogy to the POT case, the results in the latter paper can be simply reformulated in terms of a suitable second order condition (see Section~\ref{sec:secor} below for details). Based on these results, an estimator for the Pickands dependence function can then shown to have asymptotic variance of order $k_n^{\scs -1}=(n/r_n)^{-1}$, while the bias is again typically governed by a second order parameter $\rho_{\BM}<0$ and has order $r_n^{\rho_{\BM}}$. Similar calculations as in the preceding paragraph show that the best possible MSE is of order $n^{2\rho_\BM/(1-2\rho_\BM)}$. }

As indicated by the above discussion, ``best'' convergence rates for the BM and POT approaches depend on the second order parameters in their respective second order conditions. This motivates to study the relationship between the two types of second order conditions. Our first major contribution is to show that {a natural} POT second order condition, in case $\rho_\POT \in {(-1,0]}$, implies {a natural} BM second order condition with $\rho_\BM = \rho_\POT$, and vice versa. As a consequence, if ${\rho_\BM =\rho_\POT \in (-1,0)}$, the best attainable rates for POT and BM estimators coincides. The situation changes when $\rho_\POT < -1$, in which case we obtain that under mild additional conditions $\rho_\BM = \max(\rho_\POT, -1)$; similarly we prove that typically $\rho_\BM < -1$ implies $\rho_\POT = \max(\rho_\BM, -1)$. This identifies scenarios in which either BM or POT estimators can attain better rates of convergence. Finally, when $\rho_\POT = -1$ both $\rho_\BM = -1$ and $\rho_\BM < -1$ is possible (and vice versa), and additional conditions to verify which of the two cases occurs are provided. 
Note that a similar relationship between second order parameters in the univariate case has been worked out in \cite{DreDehLi03}. 

As a second major contribution, we provide a detailed analysis of second order conditions (BM and POT) for the class of Archimax copulas \citep{ChaFouGen14}. Simple sufficient conditions are formulated in terms of the Archimedean generator associated with such copulas. In particular, we show that the class of Archimax copulas copulas contains examples where either the POT or BM method can lead to faster convergence rates. This is also illustrated in a small finite-sample simulation study.

The remaining part of this article is organized as follows: in Section~\ref{sec:secor}, we introduce the second order conditions of interest and work out the connections between the two, including the above mentioned main result.
In Section~\ref{sec:ex}, we work out details in two particular examples: the general class of Archimax copulas and outer power transforms of the Clayton copula. In Section~\ref{sec:est}, we illustrate the consequences for the rate of convergence of respective estimators, both by theoretical means and by a simulation study.


\section{Second Order Conditions for the BM and the POT approach}
\label{sec:secor}

Let $(\bm X_t)_{t\in\N}$ denote an i.i.d.\ sequence of $d$-variate random vectors $\bm X_t=(X_{t1}, \dots, X_{td})$ with joint cumulative distribution function (c.d.f.) $F$ and continuous marginal c.d.f.s $F_1, \dots, F_d$. Let $C$ denote the associated unique copula. For integer $r \in\N$ and $j=1, \dots, d$, let $M_{1:r,j} = \max_{t=1}^r X_{tj}$ denote the maximum over the first $r$ observations in the $j$th coordinate, and let $\bm M_{1:r} = (M_{1:r,1}, \dots, M_{1:r,d})$. By independence, $\bm M_{1:r}$ has joint c.d.f.\ $F_r$ and copula~$C_r$, defined as $F_r(x) = F(x)^r$ and $C_r(\bm u) = C(\bm u^{1/r})^r$, where $\bm u^{s} = (u_1^s, \dots, u_d^s)$. 

We assume that $C$ lies in the copula domain of attraction of some extreme-value copula $C_\infty$, that is
\begin{align} \label{eq:dombm}
C_\infty(\bm u) = \lim_{r\to\infty} C_r(\bm u) = \lim_{r\to\infty} C(\bm u^{1/r})^r , \qquad \bm u \in [0,1]^d.
\end{align}
Hence, {$C_\infty(\bm u^{1/s})^s=C_\infty(\bm u)$ for all $s>0$ and $\bm u \in [0,1]^d$ and}
\[
C_\infty(\bm u) = \exp\{- L(-\log u_1, \dots, -\log u_d) \}, \quad \bm u \in [0,1]^d,
\]
for some \textit{stable tail dependence function} $L:[0,\infty]^d \to [0,\infty]$ satisfying
\begin{enumerate}
\item $L$ is homogeneous: $L(s\, \cdot) = sL(\cdot)$ for all  $s>0$;
\item $L(\bm e_j)=1$ for $j=1, \dots, d$, where $\bm e_j$ denotes the $j$th unit vector;
\item $\max(x_1, \dots, x_d) \le L(\bm x) \le x_1+ \dots + x_d$ for all $\bm x\in [0,\infty)^d$;
\item $L$ is convex;
\end{enumerate}
see, e.g., \cite{BeiGoeSegTeu04}.
By Taylor expansions, the assumption in \eqref{eq:dombm} is equivalent to assuming that 
\begin{align} \label{eq:dompot}
\lim_{t \to \infty} t \{ 1- C(\bm 1- \bm x / t)\} = L(\bm x), \qquad \bm x \in [0,\infty]^d,
\end{align}
where the copula $C$ is naturally extended to a c.d.f.\ on $[-\infty,\infty]^d$.
Note that the convergence is necessarily uniform on $[0,T]^d$, for any fixed $T>0$, by Lipschitz-continuity  of $C$ and $L$. 
Consider the following natural second order conditions.

\begin{definition}[Second order conditions] \label{def:secor}
Let $C$ be a copula satisfying one of the equivalent limit relations in \eqref{eq:dombm} or \eqref{eq:dompot}.
\begin{compactenum}[$(SO)_{\BM,d} \quad $]
\renewcommand{\theenumi}{$\mathrm{(SO)_\POT}$}
\renewcommand{\labelenumi}{\theenumi}
\item\label{cond:sopot}  
Suppose there exists a positive function $\alpha_\POT:(0,\infty) \to (0,\infty)$ with $\lim_{t\to\infty } \alpha_\POT(t) = 0$ and a non-null function $S_\POT$, such that,
\begin{equation*} 
\lim_{t \to \infty}  \frac{t \{ 1- C(\bm 1- \bm x / t)\} - L(\bm x)}{\alpha_{\POT}(t)}  = S_\POT(\bm x), \qquad \bm x \in [0,\infty)^d,
\end{equation*}
uniformly on $[0,T]^d$, for any fixed $T>0$.
\renewcommand{\theenumi}{$\mathrm{(SO)_{\BM,d}}$}
\renewcommand{\labelenumi}{\theenumi}
\item \label{cond:sobmbs} 
Suppose there exists a positive sequence $\alpha_{\BM,d}: \N \to (0,\infty)$ with $\lim_{r\to\infty } \alpha_{\BM,d}(\floor{r}) = 0$ and a non-null function $S_{\BM,d}$, such that,
\[
\lim_{r \to \infty}  \frac{C_{\floor{r}}(\bm u) - C_\infty(\bm u)}{\alpha_{\BM,d}(\floor{r})}  = S_{\BM,d}(\bm u), \qquad \bm u \in [0,1]^d,
\]
uniformly on $[\delta,1]^d$ for each $\delta>0$.
\renewcommand{\theenumi}{$\mathrm{(SO)_\BM}$}
\renewcommand{\labelenumi}{\theenumi}
\item \label{cond:sobm} 
Suppose there exists a positive function $\alpha_\BM:(0,\infty) \to (0,\infty)$ with $\lim_{r\to\infty } \alpha_\BM(r) = 0$ and a non-null function $S_\BM$, such that,
\begin{equation*} 
\lim_{r \to \infty}  \frac{C(\bm u^{1/r})^r  - C_\infty(\bm u)}{\alpha_{\BM}(r)}  = S_\BM(\bm u), \qquad \bm u \in [0,1]^d,
\end{equation*}
uniformly on $[\delta,1]^d$ for each $\delta>0$.
\end{compactenum}
\end{definition}

{Condition~\ref{cond:sobmbs}, with the additional requirement that the convergence be uniform on $[0,1]^d$, can be applied to the results in \cite{BucSeg14} to obtain an explicit rate of the bias term for the empirical copula of block maxima (see also Section~\ref{sec:est} below for details).} We will show {in Section~\ref{subsec:prop}} below that Condition~\ref{cond:sobmbs} is actually equivalent to {the seemingly stronger Condition~\ref{cond:sobm}} (with $S_\BM=S_{\BM,d}$, Lemma~\ref{lem:sobdtosob}) and that further the convergence in \ref{cond:sobm} must in fact be uniform on $[0,1]^d$ {(Lemma~\ref{lem:unif})}.
Finally, note that Condition~\ref{cond:sopot} was imposed in \cite{FouDehMer15}, among others.

\subsection{Some simple properties of the second order conditions}
\label{subsec:prop}

{The auxiliary functions $\alpha_m$, $m\in \{\BM, \POT\}$, in the second order conditions are necessarily regularly varying and imply a homogeneity property of the limit function $S_m$. See also \cite{FouDehMer15} for part (i) of the following lemma.}

\begin{lemma} \label{lem:reg}
(i)
Suppose that \ref{cond:sopot} is met. Then there exists $\rho_\POT \le 0$ such that $\alpha_\POT$ is regularly varying of order $\rho_\POT$.  As a consequence, $S_\POT$ is homogeneous of order $1-\rho_\POT$, that is,
\[
S_\POT(s\bm x) = s^{1-\rho_\POT}S_\POT(\bm x)
\] 
for all $s>0, \bm x \in [0,\infty)^d$.  \smallskip

\noindent
(ii) Suppose that \ref{cond:sobm} is met. Then there exists $\rho_\BM \le 0$ such that $\alpha_\BM$ is regularly varying of order $\rho_\BM$.  As a consequence,  
\[
\frac{S_\BM(\bm u^s)}{C_\infty(\bm u^s)} = s^{1-\rho_\BM} \frac{S_\BM(\bm u)}{C_\infty(\bm u)}
\]
for all $s>0, \bm u\in(0,1]^d$. 
\end{lemma}

Note that the latter display also implies a growth condition on $S_\BM$ when one coordinate approaches zero: with the constant $K_\BM=e^d \sup_{\bm v \in [e^{-1},1]} |S_\BM(\bm v)|$, which is independent of $\bu$ but can depend on $S_\BM$, we have
\begin{align} \label{eq:gsb}
|S_\BM(\bm u)| \le K_\BM  u_\wedge (-\log u_\wedge)^{1+|\rho_\BM|}, \qquad u_\wedge = \min(u_1, \dots, u_d),
\end{align}
for all $\bm u \in [0,1]^d$, with  the upper bound to be interpreted as zero if $u_\wedge=0$. Indeed, for all $\bm x\ge0$ with $x_\vee = \max(x_1, \dots, x_d) \in(0,\infty)$ we have
\[
|S_\BM(e^{-\bm x})| = |S_\BM( (e^{-\bm x/{x_\vee}})^{x_\vee})|  = x_\vee^{1+|\rho_\BM|} \frac{|S_\BM(e^{-\bm x/{x_\vee}})|}{C_\infty(e^{-\bm x/{x_\vee}})} C_\infty(e^{-\bm x})
\le
K_\BM x_\vee^{1+|\rho_\BM|} e^{-x_\vee},
\]
since $e^{-\bm x/x_\vee} \in[e^{-1},1]^d$ and since $\prod_{i=1}^d u_i \le C_\infty(\bm u) \le u_\wedge$.

\begin{proof}[Proof of Lemma~\ref{lem:reg}] We only consider assertion (ii) and for notational brevity,  we omit the index $\BM$ at all instances throughout the proof. 

By Theorem B.1.3 in~\cite{DehFer06}, regular variation of $\alpha$ follows if we prove that there exists $\mathcal{X} \subset (0,\infty)$, a measurable set of positive Lebesgue measure, such that, for all $x \in \mathcal{X}$, $\alpha(rx)/{\alpha(r)}$ converges, for $r \to \infty$, to a finite, positive function of $x$. Pick a point $\bu\in (0,1)^d$ with $S(\bu) \neq 0$ and let $\mathcal{X}$ denote a neighborhood of $1$ specified below.

For $r,x>0$, we may write, by max-stability of $C_\infty$,
\begin{align*}
\frac{\alpha(rx)}{\alpha(r)} 
= 
\frac{
	\left\{ \frac{C((\bm u^{1/x})^{1/r})^{rx}  - C_\infty(\bm u^{1/x})^{x}}{\alpha(r)} \right\}
	}{
	\left\{ \frac{C(\bm u^{1/(rx)})^{rx}  - C_\infty(\bm u)}{\alpha(rx)} \right\}
},
\end{align*}
for some arbitrary point $\bm u \in (0,1]^d$ such that $S(\bm u) \ne 0$.
The denominator converges to $S(\bm u)$. By the mean-value theorem, the numerator is equal to
\[
\{ xC_\infty(\bm u)^{1-1/x}+o(1) \}  \left\{ \frac{C((\bm u^{1/x})^{1/r})^{r}  - C_\infty(\bm u^{1/x})}{\alpha(r)} \right\}, \quad r \to \infty,
\]
which converges to $xC_\infty(\bm u)^{1-1/x} S(\bm u^{1/x})$. By continuity of $S$, the latter limit is positive for all $x$ in a sufficiently small neighborhood of $1$. 

The assertion regarding $S_\BM$ follows from elementary calculations.
\end{proof}

\begin{lemma}\label{lem:sobdtosob}
If \ref{cond:sobmbs} is met, then \ref{cond:sobm} holds with $\alpha_{\BM}(r) = \alpha_{b,d}(\floor{r})$ and $S_{\BM} = S_{\BM, d}$. 
\end{lemma}

\begin{proof}  
Throughout the proof, we omit the index $\BM$ at all instances. For $\bm u \in [0,1]^d$, let $\bm u_r := \bm u^{\floor{r}/r}$ and note that $\bu_r \to \bu$ uniformly on $[\delta,1]^d$. As explained below, the following expansion, which implies the assertion of the lemma, holds uniformly in $\bu \in [\delta,1]^d$:
\begin{align*}
C(\bu^{1/r})^r &= \Big\{C(\bu_r^{1/\floor{r}})^{\floor{r}} \Big\}^{r/\floor{r}} = \Big\{C_{\floor{r}}(\bu_r) \Big\}^{r/\floor{r}}
\\
&= \Big\{C_\infty(\bu_r) + \alpha_{d}(\floor{r})S_{d}(\bu_r) + o(\alpha_{d}(\floor{r})) \Big\}^{r/\floor{r}}
\\
&\stackrel{(a)}= \Big\{C_\infty(\bu_r) + \alpha_{d}(\floor{r})S_{d}(\bu) + o(\alpha_{d}(\floor{r})) \Big\}^{r/\floor{r}}
\\
&\stackrel{(b)}= C_\infty(\bu_r)^{r/\floor{r}} + \frac{r}{\floor{r}} \{ C_\infty(\bu_r) + o(1) \} ^{r/\floor{r}-1}  \Big\{ \alpha_{d}(\floor{r})S_{d}(\bu) + o(\alpha_{d}(\floor{r})) \Big\}
\\
&\stackrel{(c)}= C_\infty(\bu) + \alpha_{d}(\floor{r})S_{d}(\bu) + o(\alpha_{d}(\floor{r})).
\end{align*} 
Explanations: (a) is a consequence of uniform continuity of $S_{d}$. (b) follows by a Taylor expansion; note that $C_\infty(\bm u_r)$ is bounded away from $0$ uniformly in $\bm u\in[\delta,1]^d$ and $r\in\N$. This latter fact together with the fact that $r/\floor{r}$ converges to 1 implies (c).
\end{proof}

\begin{lemma} \label{lem:unif}
Let \ref{cond:sobm} be met. Then the convergence 
\[
\lim_{r \to \infty}  \frac{C(\bm u^{1/r})^r  - C_\infty(\bm u)}{\alpha_{\BM}(r)}  = S_\BM(\bm u), \qquad \bm u \in [0,1]^d,
\]
is uniform on $[0,1]^d$.
\end{lemma}

\begin{proof}  
Write $f_r(\bm u) = C(\bm u^{1/r})^r  - C_\infty(\bm u)$ and omit the index $\BM$ at $S_\BM$ and $\alpha_\BM$. Recall that $\alpha \in \RV_\rho$ with $\rho\le 0$ by Lemma~\ref{lem:reg}, and define $\gamma_r = r^{-(1+|\rho|)}=o(\alpha(r))$. Since $C(\bm u),C_\infty(\bm u) \le u_\wedge=\min(u_1, \dots, u_d)$ by the upper Fr\'echet-Hoeffding bound, we obtain that
\[
\sup_{\bm u \in [0,1]^d \setminus [\gamma_r,1]^d} 
\Big| \frac{f_r(\bm u)}{\alpha(r)} - S(\bm u) \Big| 
\le
\sup_{\bm u \in [0,1]^d \setminus [\gamma_r,1]^d}   
\frac{2 u_\wedge }{\alpha(r)} + |S(\bm u)|  =o(1)
\]
by \eqref{eq:gsb}. It is hence sufficient to show that the claimed convergence is uniform in $\bm u\in[\gamma_r,1]^d$. Suppose this is not the case. Then there exists $\eps>0$ and sequences $r_n \to \infty, \bu_n \in [\gamma_{r_n},1]^d$ with 
\[
\Big| \frac{f_{r_n}(\bm u_n)}{\alpha(r_n)} - S(\bm u_n) \Big|  \ge 2 \eps \quad \forall n.
\]
Here, the sequence $\bu_n$ must satisfy $(\bu_n)_\wedge \to 0$: indeed, for any $\eta > 0$, there exists $n_0$ with $\sup_{\bu \in [\eta,1]^d}|f_{r_n}(\bm u)/{\alpha(r_n)} - S(\bm u)| < \eps$ for all $n \geq n_0$, which implies that $(\bu_n)_\wedge < \eta $ for all $n \geq n_0$. 

By \eqref{eq:gsb} and since $(\bu_n)_\wedge \to 0$ we have $S(\bm u_n) = o(1)$. As a consequence, we may without loss of generality assume that
\begin{align} \label{eq:cont}
\Big| \frac{f_{r_n}(\bm u_n)}{\alpha(r_n)} \Big| \ge  \eps \quad \forall n.
\end{align}
Further, by \eqref{eq:gsb}, we may choose $\delta\in(0,1)$ such that $|S(\bm u)| \le \eps$ for all $\bm u$ with $u_\wedge \le \delta$. 

Next, note that $|\log \delta|/|\log ((\bu_n)_\wedge) | = (\log \delta)/\log ((\bu_n)_\wedge) $ and define
\[
 \bm v_n= \bm u_n^{s_n}, \qquad s_n= \frac{|\log \delta|}{|\log ((u_n)_\wedge) |} \ge \frac{|\log \delta|}{|\log (\gamma_{r_n}) |} = \frac{|\log \delta|}{1+|\rho|} \times  \frac1{\log r_n},
\]
so that  $s_n \to 0$ and 
\[
(\bm v_n)_\wedge = \min(v_{n,1}, \dots, v_{n,d}) =  ((\bu_n)_\wedge)^{s_n} = \delta
\]
and thus $\bm v_n\in [\delta,1]^d$.
Hence, by the mean value theorem and the Fr\'echet-Hoeffding bounds, 
\begin{align*}
|f_{r_n}(\bm u_n)|
=
|f_{r_n}(\bm v_n^{1/s_n}) |
&= 
\big|\{ C(\bm v_n^{1/(r_ns_n)})^{r_ns_n}\}^{1/s_n} - C_\infty(\bm v_n)^{1/s_n}\big| \\
&\le 
\frac1{s_n} \delta^{1/s_n - 1} |f_{r_ns_n}(\bm v_n)| 
\le 
\frac1{s_n} \delta^{1/s_n - 1} \sup_{\bm v \in [\delta, 1]^d}|  f_{r_ns_n}(\bm v)|.
\end{align*}
Further, recall the Potter bounds (e.g., Proposition B.1.9(5) in \citealp{DehFer06}): since $\alpha\in \RV_\rho$, there exists $r_0> 0$ such that
\[
\frac{\alpha(rx)}{\alpha(r)} \le 2 x^{-(1+|\rho|)} \qquad \forall\ x \in(0,1], r \ge r_0: rx\ge r_0.
\]
The preceding  two displays (note that $r_ns_n \gtrsim r_n/\log r_n \to \infty$) imply that, for sufficiently large $n$,
\begin{align*}
\Big| \frac{f_{r_n}(\bm u_n)}{\alpha(r_n)} \Big|  
&\le 
2 {s_n}^{-(2+|\rho|)} \delta^{1/s_n - 1}   \sup_{\bm v \in [\delta, 1]^d} \Big|  \frac{f_{r_ns_n}(\bm v)}{\alpha(r_ns_n)}\Big|  \\
&= 
2 {s_n}^{-(2+|\rho|)} \delta^{1/s_n - 1} \Big\{  \sup_{\bm v \in [\delta, 1]^d} |S(\bm v)| + o(1) \Big\}.
\end{align*}
The upper bound converges to $0$, which yields a contradiction to \eqref{eq:cont}.
\end{proof}

\subsection{The relationship between \ref{cond:sobm} and \ref{cond:sopot}.}

This section contains the first main result of the paper 
on the relationship between the two second order conditions \ref{cond:sobm} and \ref{cond:sopot}. 
Depending on the speed of convergence of $\alpha_\M$, we will occasionally need the following functions
\begin{align}
\label{eq:gamdir1} 
\Gamma_1(\bm x) 
&=
\partial_{\bm x^2}L(\bm x) =
 \lim_{r \to \infty} r\Big\{L\Big(\bm x + \frac{\bm x^2}{r}\Big) - L(\bm x)\Big\}, \qquad \bm x \in [0,\infty)^d, \\
 \label{eq:gamdir2}
 \Gamma_2(\bm x) 
&=
- \partial_{-\bm x^2}L(\bm x) =
 \lim_{r \to \infty} r\Big\{L(\bm x) - L\Big(\bm x - \frac{\bm x^2}{r}\Big)\Big\}, \qquad \bm x \in [0,\infty)^d.
\end{align}
Note that both limits necessarily exist for all $\bm x \in [0,\infty)^d$: indeed, by convexity of $L$, the difference quotients inside the limits  are monotone functions of $r$ (see Theorem 23.1 in \citealp{Roc70}) and, by Lipschitz continuity of $L$, they are uniformly bounded.  Furthermore, the functions $\Gamma_1, \Gamma_2$ must be homogeneous of order $2$, satisfy $0\le \Gamma_\ell(\bm x) \le \sum_{j=1}^d x_j^2$ and may be discontinuous. For a class of examples regarding the last assertion, consider the case $d=2$ with Pickands dependence function $A(t) = L(1-t,t)$. A straightforward but tedious calculation utilizing the homogeneity of $L$ shows that 
\[
L\Big(1-t + \frac{(1-t)^2}{r},t + \frac{t^2}{r}\Big) = A\Big(t+\frac{t(1-t)(2t-1)}{r}\Big) + A(t)\frac{1-2t+2t^2}{r} + O(r^{-2}). 
\]
Hence, we have
\[
\Gamma_1(1-t,t) = A(t)(1-2t+2t^2) + \lim_{r \to \infty} r\Big\{A\Big(t+\frac{t(1-t)(2t-1)}{r}\Big) - A(t)\Big\}. 
\]
For $t \in (1/2,1)$ this limit is continuous if and only if $s \mapsto \lim_{h \downarrow 0} (A(s+h) - A(s))/h$ is continuous at $t$ which can fail for piecewise linear functions $A$. 

For the general result to come, the convergences  in \eqref{eq:gamdir1} and \eqref{eq:gamdir2} must be uniform on $[0,T]^d$. Sufficient conditions are formulated in the next lemma, where we also provide a representation of $\Gamma_\ell$ in terms of the partial derivatives of $L$.

\begin{lemma}\label{lem:gamma}
(i) If the limit $\Gamma_1$ in \eqref{eq:gamdir1} is continuous, then the convergence is uniform on $[0,T]^d$ for any $T>0$. The same assertion holds for $\Gamma_2$ and the convergence in \eqref{eq:gamdir2}.

\smallskip
\noindent (ii) If, for all $j=1, \dots, d$, the first order partial derivative $\dot L_j$ of $L$ exists and is continuous on $\{\bm x \in [0,\infty)^d:x_j>0\}$, then the convergences in \eqref{eq:gamdir1}  and \eqref{eq:gamdir2} are uniform on $[0,T]^d$ for any $T>0$ and we have $\Gamma_1=\Gamma_2=\Gamma$ with
\[
\Gamma(\bm x) =  \sum_{j: x_j>0} x_j^2\dot L_j(\bm x).
\]
\end{lemma}

\begin{proof} 
The assertion in (i) is a consequence of Dini's theorem: 
\[
\Gamma_r(\bm x) := r \{ L(\bm x + \bm x^2/r) -L(\bm x ) \}
\] 
is a continuous function of $\bm x$ and a monotone function of $r$ converging point-wise to a limit $\Gamma$ which is continuous by assumption.   

For the proof of (ii), apply the mean value theorem to write 
\[
\Gamma_r(\bm x) = \sum_{j: x_j>0} x_j^2 \dot L_j(\bm x + o(1)),
\] 
where the $o$-term is uniform on $[0,T]^d$. As a consequence,
\[
\big| \Gamma_r(\bm x) - \Gamma(\bm x) \big| 
\le
\sum_{j: x_j>0} x_j^2 | \dot L_j(\bm x+o(1)) - \dot L (\bm x)|.
\]
It is now sufficient to show uniform convergence to zero for each summand on the right-hand side separately. For arbitrary $\eps>0$, decompose $[0,T]^d$ into $\{\bm x\in[0,T]^d: x_j < \eps\}$ and $\{\bm x\in[0,T]^d: x_j \ge \eps\}$. On the first set, we have $x_j^2 | \dot L_j(\bm x+o(1)) - \dot L (\bm x)|< \eps^2$ by boundedness of $\dot L_j$. On the second set, the function $\dot L_j$ is uniformly continuous, whence  $x_j^2 | \dot L_j(\bm x+o(1)) - \dot L (\bm x)| = o(1)$ uniformly.\end{proof}

The next two theorems are the main results of this section, and provide simple conditions that allow to derive \ref{cond:sobm} from \ref{cond:sopot} and vice versa. The most important consequence is that, under minimal extra conditions, \ref{cond:sobm} with second order parameter $\rho_\BM \ne -1$ implies  \ref{cond:sopot} with second order parameter $\rho_\POT=\max(\rho_\BM, -1)$, and vice versa.
Let
\[
L_\vee (\bm x) = \max(x_1, \dots, x_d), \quad \bm x\in [0,\infty)^d
\]
denote the stable tail dependence function corresponding to perfect tail dependence.


\begin{theorem} \label{theo:bp}
(a) Suppose that \ref{cond:sopot} is met with $\alpha_\POT$ regularly varying of order $\rho_\POT\le 0$ and assume that $\lim_{r \to \infty} 2r\alpha_\POT(r) = c_\POT\in [0, \infty]$.
\begin{enumerate}
\item[(i)] If $c_\POT = \infty$, then \ref{cond:sobm} holds with $\alpha_\BM \equiv \alpha_\POT$ and $S_\BM(e^{-\bm x})= -C_\infty(e^{-\bm x})  S_\POT(\bm x) $.
\item[(ii)] If $c_\POT = 0$ and $L\ne L_\vee$, then \ref{cond:sobm} holds if and only if $\Gamma_2$ in \eqref{eq:gamdir2} is continuous. We may choose $\alpha_\BM(r) = (2r)^{-1}$ and $S_\BM(e^{-\bm x}) = C_\infty(e^{-\bm x}) \{ \Gamma_2(\bm x)  - L^2(\bm x) \}$.
\item[(iii)] If $c_\POT \in (0,\infty)$, $\Gamma_2$ is continuous, and $S_\BM$ defined below is not the null function, then \ref{cond:sobm} is met with %
\[
\alpha_\BM(r) = \alpha_\POT(r) + \frac1{2r},
\]
and with
\[
S_{\BM} (e^{- \bm x}) = C_\infty(e^{-\bm x}) \big\{ \lambda_\POT \{\Gamma_2(\bm x)  - L^2(\bm x) \} -  (1-\lambda_\POT) S_\POT(\bm x) \big\},  \qquad \lambda_\POT=\frac{1}{1+c_\POT}.
\] 
\end{enumerate}

\medskip
\noindent
(b) Suppose that \ref{cond:sobm} is met with $\alpha_\BM$ regularly varying of order $\rho_\BM\le 0$ and assume that $\lim_{r \to \infty} 2r\alpha_\BM(r) = c_\BM\in [0, \infty]$. 
\begin{enumerate}
\item[(i)] If $c_\BM = \infty$, then \ref{cond:sopot} holds with $\alpha_\POT \equiv \alpha_\BM$ and $S_\POT(\bm x) = - S_\BM(e^{-\bm x})/C_\infty(e^{-\bm x})$.
\item[(ii)] If $c_\BM = 0$ and $L\ne L_\vee$, then \ref{cond:sopot} holds if and only if $\Gamma_1$ in \eqref{eq:gamdir1} is continuous. We may choose $\alpha_\POT(r) = (2r)^{-1}$ and $S_\POT(\bm x) = \Gamma_1(\bm x)  - L^2(\bm x)$.
\item[(iii)] If $c_\BM \in (0,\infty)$, $\Gamma_1$ is continuous, and $S_\POT$ defined below is not the null function, then \ref{cond:sopot} is met with %
\[
\alpha_\POT(r) = \alpha_\BM(r) + \frac1{2r},
\]
and with
\[
S_\POT(\bm x) = 
 \lambda_\BM \{\Gamma_1(\bm x)  - L^2(\bm x) \} - (1-\lambda_\BM) \frac{S_\BM(e^{-\bm x}) }{C_\infty(e^{-\bm x})} ,  \qquad \lambda_\BM=\frac{1}{1+c_\BM}. 
\]
\end{enumerate}

\end{theorem}

In other words, if one of the two conditions holds with $c_\M=\infty$, which is only possible for $\rho_m\in[-1,0]$ and necessarily the case if $\rho_\M \in( -1,0]$, then the other condition holds as well with $\rho_\BM = \rho_\POT$; in that case we can choose $\alpha_\POT = \alpha_\BM$.  

When \ref{cond:sopot} is met with $c_\POT =0$, which is only possible for $\rho_\POT\le -1$ and necessarily the case if $\rho_\POT <-1$,  the proof of Theorem~\ref{theo:bp} actually shows that the limit relation required for \ref{cond:sobm}, 
\[
\lim_{r \to \infty}  \frac{\{ C(\bm u^{1/r})^r - C_\infty(\bm u)\}}{1/(2r)} = C_\infty(\bm u)\{\Gamma_2(-\log \bm u)  - L^2(-\log \bm u)\}, \qquad \bm u \in [0,1]^d,
\]
holds point-wise in $\bm u$, and that this limit is not the zero function if and only if $L \ne L_\vee$  (i.e., only $\rho_\BM=-1$ is possible). However, even if the limit is non-zero, \ref{cond:sobm} can still fail due to a lack of uniformity. If $L = L_\vee$, a close look at the proofs (keeping track of higher order terms) reveals that \ref{cond:sopot} implies \ref{cond:sobm} provided that $\lim_{r \to \infty} r^2 \alpha_\POT(r) = \infty$. In that case we can choose $\alpha_\BM \equiv \alpha_\POT$ and $S_\BM(e^{-\bm x}) = - S_\POT(\bm x)C_\infty(e^{-\bm x})$. If $\lim_{r \to \infty} r^2 \alpha_\POT(r) < \infty$, even higher order expansions along similar lines are possible. Similar comments apply to case (b) in Theorem~\ref{theo:bp}.

When \ref{cond:sopot} is known to hold with $c_\POT\in(0,\infty)$, which is only possible for $\rho_\POT = -1$, then both $\rho_\BM = -1$ and $\rho_\BM < -1$ is possible (and vice versa), and additional case-by-case calculations are necessary. As a matter of fact, the function $S_\BM$ defined in Theorem~\ref{theo:bp}(a), and likewise $S_\POT$ in (b), may be zero. Indeed, in Section~\ref{sec:est} we provide an example where $\rho_\POT = -1, \rho_\BM = -2$ (and vice versa). Starting from $\rho_\POT = -1$ we see that $S_\BM$ defined in Theorem~\ref{theo:bp} (a) must be zero since otherwise we would have $\rho_\BM = -1$.  A further, more direct calculation, is possible for the bivariate independence copula. A simple calculation shows that  \ref{cond:sopot} is met with $\alpha_\POT(r) = 1/(2r)$ and $S_\POT(x,y) = -2xy$. Further, $\lambda_\POT=1/2$, $L(x,y) = x+y$ and $\Gamma_2(x,y) =\Gamma(x,y) = x^2+y^2$, which implies that the function $S_\BM$ in part (iii) of (a) is the null-function. Hence, Theorem~\ref{theo:bp} does not make any assertion about whether \ref{cond:sobm} holds. In fact, since the independence copula is an extreme-value copula, the numerator on the right-hand side of \ref{cond:sobm} is actually zero, so \ref{cond:sobm} is not met at all as the limit cannot be nonzero.

\begin{proof}[Proof of Theorem~\ref{theo:bp}] It suffices to consider the case $\bm u = e^{-\bm x} \in (0,1]^d$, that is, $\bm x = - \log \bm u \in[0,\infty)^d$. Subsequently, let $\delta\in(0,1)$ and $T>1$ be arbitrary, but fixed. All $o$- and $O$-notations within the proof are to be understood uniformly in $[\delta,1]^d$ or $[0,T]^d$, depending on whether $\bm u$ or $\bm x$ is involved.

Now, either \ref{cond:sobm}  or \ref{cond:sopot} implies \eqref{eq:dombm} (and the convergence is uniform on $[\delta, 1]^d$) and \eqref{eq:dompot} (and the convergence is uniform on $[0,T]^d$),  and we begin by collecting two properties implied by the latter two limit relations. 
First of all, since
\[
r \log C(\bm u^{1/r})  - \log C_\infty(\bm u) =o(1)
\]
by~\eqref{eq:dombm} and  $C_\infty(\bm u) > 0$, a Taylor expansion of the exponential function implies that 
\begin{align*}
	\frac{C(\bm u^{1/r})^r  - C_\infty(\bm u)}{\alpha_{\BM}(r)} 
	&=
	\frac{r \log C(\bm u^{1/r})  - \log C_\infty(\bm u)}{\alpha_{\BM}(r)}  \\
	&\hspace{1cm}\times 
	\Big[
		C_\infty(\bm u) + \frac12 \big\{ C_\infty(\bm u) + o(1) \big\}  
		\Big\{ r \log C(\bm u^{1/r})  - \log C_\infty(\bm u) \Big\} 
	\Big] \\
	&= 
	\frac{r \log C(\bm u^{1/r}) - \log C_\infty(\bm u)}{\alpha_{\BM}(r)} \{ C_\infty(\bm u) + o(1) \}.
\end{align*}
As a consequence, the convergence in \ref{cond:sobm} for $\bm  u \in (0,1]^d$ is actually equivalent to 
\begin{align} \label{eq:sobm2}
\lim_{r \to \infty} \frac{r \log C(\bm u^{1/r})  - \log C_\infty(\bm u)}{\alpha_{\BM}(r)} = \frac{S_\BM(\bm u) }{C_\infty(\bm u)},
\end{align}
and if either the convergence in \ref{cond:sobm} or in \eqref{eq:sobm2}  is uniform on $[\delta,1]^d$, then so is the other. Second, 
since $C$ is Lipschitz continuous, 
we obtain that, by~\eqref{eq:dompot}, 
\begin{align} \label{eq:r1c}
r \{ 1-C(e^{-\bm x/r})\} = r \{ 1-C(1-\bm x/r)\} +  O(r^{-1}) = L(\bm x) + o(1).
\end{align}

Let us now prove the assertions in (a). As argued above, it suffices to show \eqref{eq:sobm2}.  Now, the second order condition \ref{cond:sopot} can be rewritten as
\[
r \{ 1-C(1-\bm x/r) \} =  L(\bm x) + \alpha_\POT(r) \{ S_\POT(\bm x) + o(1) \}.
\]
Let $\bm y_r = r(1-e^{-\bm x/r}) - \bm x = - \bm x^2/(2r) + O(r^{-2})$, and note that $e^{-\bm x/r}=1-(\bm x + \bm y_r)/r$. We may thus write
\begin{align*}
r \{ 1-C(e^{-\bm x/r}) \} 
&=
r \{ 1-C(1-(\bm x+ \bm y_r)/r) \}  \\
&=
 L(\bm x + \bm y_r ) + \alpha_\POT(r) \{ S_\POT(\bm x+ \bm y_r) + o(1) \}
\\
&= L(\bm x - \bm x^2/(2r)) + \alpha_\POT(r) \{ S_\POT(\bm x) + o(1) \} + O(r^{-2})
\end{align*}
by uniform continuity of $S_\POT$ and Lipschitz continuity of $L$. As a consequence, by a Taylor expansion of the logarithm,
\begin{align}
- r \log C(\bm u^{1/r})  
&= 
- r \log(1+C(e^{-\bm x/r}) - 1) \nonumber \\
&=
 r \{ 1 - C(e^{-\bm x/r}) \} + [ r \{ 1- C(e^{-\bm x/r})) ]^2/(2r)  + O(r^{-2}) \label{eq:tayl} \\
&=
L(\bm x - \bm x^2/(2r)) + \alpha_\POT(r) \{ S_\POT(\bm x) + o(1) \} + \{ L(\bm x)^2 + o(1)\}/(2r) + O(r^{-2}),\nonumber
\end{align}
where we have used \eqref{eq:r1c} in the last equality.
Hence, we have
\begin{align*}
r \log C(\bm u^{1/r})  - \log C_\infty(\bm u)
&=
\frac{1}{2r}\Big[ 2r \{ L(\bm x) - L(\bm x - \bm x^2/(2r)) \}  - L(\bm x)^2\Big] -\alpha_\POT(r) S_\POT(\bm x) \\
&\qquad  + o(\alpha_\POT(r) + r^{-1}).
\end{align*}
The claim in (i) now follows from boundedness of the term in square brackets, which is a consequence of Lipschitz continuity of $L$. The claim in (iii) follows after some elementary calculations taking into account that the convergence in~\eqref{eq:gamdir2} is uniform by Lemma~\ref{lem:gamma}. For the proof of (ii), note that the latter display implies that
\begin{align*}
\lim_{r \to \infty}  2r \{r \log C(\bm u^{1/r})  - \log C_\infty(\bm u)\}
&= - L(\bm x)^2 + \lim_{r \to \infty} 2r \{ L(\bm x) - L(\bm x - \bm x^2/(2r)) \}   
\\
& = \Gamma_2(\bm x) - L^2(\bm x)
\end{align*}
point-wise in $\bm x \in [0,\infty)^d$. By \eqref{eq:sobm2}, this is equivalent to pointwise convergence in \ref{cond:sobm} with $\alpha_\BM = 1/(2r)$ and $S_\BM(e^{-\bm x}) = C_\infty(e^{-\bm x}) \{ \Gamma_2(\bm x)  - L^2(\bm x) \}$. The assertion in (ii) then follows from the facts that the convergence is uniform if and only if the convergence in~\eqref{eq:gamdir2} is uniform (which is equivalent to continuity of $\Gamma_2$ by Lemma~\ref{lem:gamma}) and that the limit is non-zero if and only if $L \ne L_\vee$. To see the latter, a simple calculation shows that  $\Gamma_2=L^2$ holds for $L=L_\vee$. On the other hand, if $\Gamma_2=L^2$, then $L^2(\bm x)= \Gamma_2(\bm x) = \sum_{j=1}^d x_j^2 \dot L_j(\bm x)$ for all $\bm x$ in the set $\mathcal C$ of points where $L$ is continuously differentiable; note that the complement of $\mathcal C$ is a Lebesgue null set by Theorem 25.5 in \citealp{Roc70}. A version of Euler's homogeneous function theorem then implies  
\[
|L(\bm x) - L^2(\bm x)| = |L(\bm x) - \Gamma_2(\bm x)| = | \sum_{j=1}^d (x_j - x_j^2) \dot L_j(\bm x) | \le \max(x_1, \dots, x_d) \sum_{j=1}^d |x_j-1|
\] 
for all $\bm x \in \mathcal C$. Taking limits along a sequence in $\mathcal C$ converging to $\bm 1$, we obtain that $L(\bm 1)=L^2(\bm 1)$ and hence $L(\bm 1)=1$, which only occurs for $L=L_\vee$. 

Let us now prove part (b) of the theorem. The assertion in \eqref{eq:sobm2}, which is equivalent to \ref{cond:sobm}, may be rewritten as
\[
-r \log C(e^{-\bm x/r})= L(\bm x) - \alpha_\BM(r) \{ S_\BM (e^{-\bm x}) / C(e^{-\bm x}) + o(1) \}.
\]
The Taylor expansion in the first two lines of \eqref{eq:tayl}, together with \eqref{eq:r1c}, allows to rewrite this as
\[
r\{ 1-C(e^{-\bm x/r}) \} = L(\bm x) - \alpha_\BM(r) \{  S_\BM (e^{-\bm x}) / C(e^{-\bm x}) + o(1) \} -  \{ L(\bm x)^2 + o(1)\}/(2r)  + O(r^{-2}).
\]
Let $\bm z_r = -r\log(1 - \bm x/r) - \bm x = \bm x^2/(2r) + O(r^{-2})$, and note that $1-\bm x/r=e^{-(\bm x+\bm z_r)/r}$. The previous display then implies
\begin{align*}
r\{1- C(1-\bm x/r)\} 
&=
r\{1-C(e^{-(\bm x+\bm z_r)/r})\} \\
&=  L(\bm x + \bm z_r) - \alpha_\BM(r) \{  S_\BM (e^{-\bm x-\bm z_r}) / C(e^{-\bm x-\bm z_r}) + o(1) \}  \\
& \hspace{2cm} -  \{ L(\bm x+ \bm z_r)^2 + o(1)\}/(2r)  + O(r^{-2}) \\
&=  L(\bm x  + \bm x^2/(2r) ) - \alpha_\BM(r) \{  S_\BM (e^{-\bm x}) / C(e^{-\bm x}) + o(1) \}  \\
& \hspace{2cm}  -  \{ L(\bm x)^2 + o(1)\}/(2r)  + O(r^{-2}), 
\end{align*}
and therefore, 
\begin{align*}
r\{1- C(1-\bm x/r)\}  - L(\bm x)
&=
\frac{1}{2r} 
\Big[ 2r \{ L(\bm x + \bm x^2/(2r)) - L(\bm x) \}  -L(\bm x)^2 \Big] 
- \alpha_\BM(r) \frac{S_\BM(e^{-\bm x})}{C(e^{-\bm x})}
\\
& \qquad + o(r^{-1}+\alpha_\BM(r)).
\end{align*}
The remaining part of the proof is now similar to the proof of (a).
\end{proof}

\section{Second Order Conditions in Particular Models} \label{sec:ex}

\subsection{Outer Power Transform of a Clayton Copula} \label{ex:opc}
For $\theta>0$ and $\beta \ge 1$, the outer power transform of a Clayton Copula is defined as 
\[
C_{\theta, \beta}(u,v) \equiv \Big[ 1+ \big\{ (u^{-\theta} - 1)^\beta + (v^{-\theta} - 1)^{\beta} \big\}^{1/\beta} \Big]^{-1/\theta}, \qquad (u,v)\in[0,1]^2,
\]
to be interpreted as zero if $\min(u,v)=0$. Note that $C_{\theta,\beta}$ is an Archimedean copula with generator $\varphi(t) = \{ (t^{-\theta}-1)/\theta\}^\beta$.
It follows from Theorem 4.1 in \cite{ChaSeg09} that $C_{\theta,\beta}$ is in the copula domain of attraction of the Gumbel--Hougaard copula with shape parameter $\beta$, that is,
\[
C_\infty(u,v) = C_{\beta}(u,v) \equiv \exp\Big[ - \big\{ (- \log u)^\beta + (-\log v)^{\beta} \big\}^{1/\beta} \Big], \qquad (u,v)\in[0,1]^2,
\]
again to be interpreted as zero if  $\min(u,v)=0$.
Moreover, by Proposition 4.3 in \cite{BucSeg14}, Condition~\ref{cond:sobm} is met with  $\rho_\BM = -1$ and
\[
\alpha_\BM(r) = (2r)^{-1} \quad \text{ and } \quad S_\BM(u,v)= \theta\, \Lambda_\BM(u,v;\beta),
\]
where
\begin{align*}
  \Lambda_\BM(u,v;\beta)  & =C_{\beta}(u,v)  \,
  \{ (x^\beta + y^\beta)^{2/\beta} - (x^\beta + y^\beta)^{1/\beta-1} (x^{\beta+1} + y^{\beta+1}) \} \\
  &=
  C_{\beta}(u,v)  L_\beta(x,y)  \,
\Big\{  L_\beta(x,y)  -\frac{x^{\beta+1} + y^{\beta+1}}{x^\beta + y^\beta} \Big\}
\end{align*}
with $x=-\log u$ and  $y = -\log v$. The constant $c_\BM$ in Theorem~\ref{theo:bp} is hence equal to $c_\BM=1$. 

By similar calculations as in the above reference, Condition~\ref{cond:sopot} could be verified from scratch. However, a much simpler calculation shows that $L_\beta$ satisfies the conditions from  Lemma~\ref{lem:gamma} (ii) and hence we obtain
\[
\Gamma_1(\bm x) = \Gamma(\bm x) =  L_\beta(x,y)   \frac{x^{\beta+1} + y^{\beta+1}}{x^\beta + y^\beta}.
\]
We may hence apply Theorem~\ref{theo:bp} to obtain that \ref{cond:sopot} is met with
$\rho_\POT=-1$ and
\[
\alpha_{\POT}(t) = t^{-1} \quad \text{ and } \quad S_{\POT}(x,y)= \frac{1+\theta}{2}\,L_\beta(x,y)  \,
  \Big\{ \frac{x^{\beta+1} + y^{\beta+1}}{x^\beta + y^\beta} -  L_\beta(x,y) \Big\}.
\]

\subsection{Archimax Copulas} \label{subsec:archimax}
Let $L_0$ be an arbitrary stable tail dependence function. Further, let $\psi:[0,\infty) \to [0,1]$ be a $d$-Archimedean generator \citep{McnNes09}, that is, a $d$-monotone function which is strictly decreasing on $[0, \inf\{x \ge 0: \psi(x) = 0\}]$ and satisfies  the conditions $\psi(0)=1$ and  $\lim_{x\to \infty} \psi(x)=0$. Here, $d$-monotonicity means that the first $d-2$ derivatives of $\psi$ exist on $(0,\infty)$ and satisfy $(-1)^j \psi^{(j)} \ge 0$ for all $j=0, \dots, d-2$, with $(-1)^{d-2} \psi^{(d-2)}$ being convex and non-increasing on $(0,\infty)$. In particular, in the bivariate case, $\psi$ must be convex. The Archimax copula associated with $(\psi, L_0)$ is defined as 
\[
C(\bm u) = \psi[ L_0\{ \psi^{-1}(u_1), \dots, \psi^{-1}(u_d) \} ], \quad \bm u \in [0,1]^d,
\]
where $\psi^{-1}$ denotes the inverse of $\psi$ on $(0,1]$ and where $\psi^{-1}(0) = \inf\{x \ge 0: \psi(x) = 0\}$, see \cite{ChaFouGen14}. By Proposition 6.1 in that reference, the copula $C$ is in the max-domain of attraction of the copula
\[
C_\infty(\bm u) = \exp\{ - L_0^\alpha( x_1^{1/\alpha}, \dots, x_d^{1/\alpha}) \}, \quad \bm u \in [0,1]^d,
\]
where $x_j=-\log u_j$, provided that the function $\kappa_{\POT}: w \mapsto 1-\psi(1/w)$ is regularly varying with index $-\alpha$ for some $\alpha \in (0,1]$. Equivalently, the function $\lambda_{\POT}: w \mapsto \psi^{-1}(1-1/w)$ is regularly varying with index $-1/\alpha$, see equation (15) in \cite{ChaFouGen14}. Finally, note that $C$ itself is an extreme-value copula provided $\psi(x)=\exp(-x)$ and that the outer-power transform of the Clayton copula considered in Section~\ref{ex:opc} is an Archimax copula with parameter $\psi(x) = (1+\theta x^{1/\beta})^{-1/\theta}$ (i.e., $\kappa_\POT \in \RV_{-1/\beta}$) and $L_0(x,y) = x+y$.

Define yet another function $\kappa_{\BM}: w \mapsto -\log \psi(1/w)$, and note that
$\kappa_{\POT} \in \RV_{-\alpha}$ if and only if $\kappa_\BM \in \RV_{-\alpha}$. It is further possible to prove that this is equivalent to the function $\lambda_\BM: w \mapsto \psi^{-1}(e^{-1/w})$ being regularly varying with index $-1/\alpha$.\footnote{Indeed, we have $\kappa_b=-\log \psi(1/\cdot) \in \RV_{-\alpha}$ iff $1/(-\log \psi(1/\cdot)) \in \RV_\alpha$. Since the latter function is necessarily strictly increasing, this implies $\{ 1/(-\log \psi(1/\cdot)) \} ^{-1} = 1/\psi^{-1}(e^{-1/\cdot}) \in \RV_{1/\alpha}$, by Proposition B.1.9(9) in \cite{DehFer06}. This implies $\lambda_b \in\RV_{-1/\alpha}$. The other direction is similar. }.
Consider the following second order strengthening for $\M \in \{\BM, \POT\}$: there exists a positive function $B_{\M}$ with $\lim_{t \to\infty} B_{\M}(t) = 0$ and a  function $h_{\kappa,\M}$ which is not of the from $cx^{-\alpha}$ for some $c\in\R$ such that
\begin{align} \label{eq:sophi1}
\lim_{t\to\infty} \frac{1}{B_{\M} (t)}  \Big\{ \frac{\kappa_\M(tx)}{\kappa_\M(t)} - x^{-\alpha}  \Big\} = h_{\kappa,\M}(x)
\end{align}
for all $x>0$.  In that case, by applying Theorem B.2.1 in \cite{DehFer06} to the functions $f(t) = t^\alpha \kappa_\M(t)$ and $a(t)=t^\alpha \kappa_\M(t) B_{\M}(t)$, the limit $h_{\kappa,\M}$ is necessarily of the form 
\begin{align} \label{eq:hkappa}
h_{\kappa,\M}(x) = c_\M x^{-\alpha} \frac{x^{\rho_\M'}-1}{\rho_\M'}
\end{align}
for some constant $c_\M\ne 0$ and $\rho_\M'\le 0$; when $\rho_\M' = 0$ the fraction $(x^{\rho_\M'}-1)/{\rho_\M'}$ should be interpreted as $\log x$. Moreover, $B_{\M}$ is regularly varying with index $\rho'_\M$. As an example, consider the continuously differentiable $2$-Archimedean generator
\[
\psi(x) = (1-x+x^2/2)\bm 1(x \in [0,1/2]) + (7/8 -x/2) \bm 1(x \in  [1/2, 14/8]),
\]
such that $\kappa_\BM, \kappa_\POT \in\RV_{-1}$, i.e, $\alpha=1$. In that case, it can be shown that \eqref{eq:sophi1} is met for $\M = \BM, \POT$ with
\[
B_{\BM}(t) = t^{-2}, \quad \rho_\BM' = -2 \quad\text{ and }\quad B_{\POT}(t) = t^{-1}, \quad \rho_\POT' = -1.
\]

\begin{proposition} \label{prop:archso}
Suppose that, for $j=1, \dots, d$, the $j$th first-order partial derivative $\dot L_{0,j}$ of $L_0$ exists and is continuous on  $\{\bm x \in [0,\infty)^d: x_j>0\}$. For completeness, define
\[
\dot L_{0,j}(\bm x) = \limsup_{h \downarrow 0} \frac{L_0(\bm x +h \bm e_j)-L_0(\bm x)}{h} \in[0,1]
\]
for $\bm x$ such that $x_j=0$. If \eqref{eq:sophi1} is met for $\M = \POT$ and with $\rho_{\POT}'<0$, then \ref{cond:sopot} is met with $\alpha_\POT(t) = B_{\POT}(t^{1/\alpha}) \in \RV_{\rho_{\POT}'/\alpha}$. Similarly, if \eqref{eq:sophi1} is met for $\M= \BM$ and with $\rho_\BM'<0$, then \ref{cond:sobm} is met with $\alpha_\BM(t) = B_{\BM}(t^{1/\alpha}) \in \RV_{\rho_{\BM}'/\alpha} $. 
\end{proposition}

\begin{proof}
The proof heavily relies on Lemma~\ref{lem:inv_short} below.
Slightly abusing notation, we write $f(\bm x) = (f(x_1), \dots, f(x_d))$, provided $f$ is a real-valued function on a subset of the extended real line.  {Depending on the context, all convergences are for $t\to \infty$ or $r\to\infty$, if not mentioned otherwise.}

Consider the case $\M=\POT$ first.
Using the homogeneity of $L_0$ and the fact that the convergence in \eqref{eq:sophi1} 
is uniform on sets $[\eps,\infty)$ (Lemma~\ref{lem:inv_short} below), we may write, using the notation $C_\POT$ defined in Lemma~\ref{lem:inv_short},
\begin{align}
\label{eq:decall1}
t\{ 1- C(1-\bm x/t) \} 
&= \nonumber
	\kappa_\POT\Big( \frac{1}{L_0( \lambda_\POT(t/\bm x) ) } \Big) 
	\Big /
	\kappa_\POT\Big( \frac{1}{ \lambda_\POT(t) } \Big) 
\\
&\stackrel{(a)}{=} \nonumber
	\kappa_\POT\bigg(\frac{1}{\lambda_\POT(t) }
		\frac{1}{L_0\big( \frac{\lambda_\POT(t/\bm x) }{\lambda_\POT(t) } \big) } \bigg) 
	\Big/
	\kappa_\POT\Big( \frac{1}{ \lambda_\POT(t) } \Big)
\\
&\stackrel{(b)}{=} \nonumber
	L_0^\alpha\Big( \frac{\lambda_\POT(t/\bm x) }{\lambda_\POT(t) } \Big) 
	+ 
	\bigg\{ h_{\kappa,\POT}\bigg(  
		\frac{1}{L_0\big( \frac{\lambda_\POT(t/\bm x) }{\lambda_\POT(t) } \big) } 
		\bigg) +o(1) \bigg\} 
	B_{\POT}\Big( \frac{1}{ \lambda_\POT(t) } \Big)  \\
&\stackrel{(c)}{=} \nonumber
	L_0^\alpha\Big( \frac{\lambda_\POT(t/\bm x) }{\lambda_\POT(t) } \Big) 
	+ 
	\bigg\{ h_{\kappa,\POT}\bigg(  
		\frac{1}{L_0\big(\bm x^{1/\alpha} \big) } 
		\bigg) +o(1) \bigg\} 
	B_{\POT}\Big( \frac{1}{ \lambda_\POT(t) } \Big) 
	\\
&\stackrel{(d)}{=} 
	L_0^\alpha\Big( \frac{\lambda_\POT(t/\bm x) }{\lambda_\POT(t) } \Big) 
	+
	\bigg\{ h_{\kappa,\POT}\bigg(  
		\frac{1}{L_0\big(\bm x^{1/\alpha} \big) } 
		\bigg) +o(1) \bigg\} C_\POT^{\rho'/\alpha}B_{\POT}(t^{1/\alpha})
\end{align}
where all $o$-terms are uniform in $\bm x\in[0,T]^d$, and where $1/0$ is interpreted as $\infty$ and converging functions on $(0,\infty)$ are naturally extended to $(0,\infty]$. Explanations: (a) follows from homogeneity of $L_0$, (b) from the fact that the convergence in \eqref{eq:sophi1} in uniform on $[\eps,\infty)$ by Lemma~\ref{lem:inv_short}, (c) from uniform continuity of $h_{\kappa,\POT}$ on $[\eps,\infty]$ and continuity of $L_0$, and (d) from \eqref{eq:B1lambda} in the proof of Lemma~\ref{lem:inv_short}. 

We are next going to show that the first summand on the right-hand side of \eqref{eq:decall1} can be written as 
\begin{align} \label{eq:con1}
L_0^\alpha\Big( \frac{\lambda_\POT(t/\bm x) }{\lambda_\POT(t) } \Big) 
= 
L_0^\alpha(\bm x^{1/\alpha})  + B_{\POT}(t^{1/\alpha}) \big\{ S_{\POT,1}(\bm x) + o(1) \big\},   
\end{align}
where the $o$-term is uniform in $\bm x\in[0,T]^d$ and where
\begin{align*}
S_{\POT,1}(\bm x) =  \alpha L_0^{\alpha-1}(\bm x^{1/\alpha})  \sum_{j=1}^d   \dot L_{0,j}(\bm x^{1/\alpha}) h_{\lambda, \POT}(x_j^{-1}) 
\end{align*}
with $S_{\POT,1}(\bm 0)$  being interpreted as $0$ and {$h_{\lambda, \POT}$ being defined in \eqref{eq:lambdalim} in Lemma~\ref{lem:inv_short} below.}

For that purpose, let $\eps>0$, and note that we may find $\delta>0$ such that $|S_{p,1}(\bm x)| \le \eps/2$ for all $\bm x \in [0, \delta]^d$. Further, note that $\bm x \mapsto L_0^\alpha(\bm x^{1/\alpha})$ is Lipschitz-continuous. {Indeed, using that 
\[
z_1^\alpha - z_2^\alpha = \alpha (z_1 -z_2) \int_0^1 \{s z_1 + (1-s) z_2\}^{\alpha-1} \diff s
\] 
for all $z_1, z_2 \ge 0$ with $z_1z_2\ne 0$, we have, for all $\bm x, \bm y$ not both being equal to $\bm 0$,
\begin{align*}
|L^\alpha(\bm x^{1/\alpha}) - L^\alpha(\bm y^{1/\alpha})| 
&=  \textstyle
\alpha |L(\bm x^{1/\alpha}) - L(\bm y^{1/\alpha})|  \int_0^1 \{s L(\bm x^{1/\alpha}) + (1-s) L(\bm y^{1/\alpha})\}^{\alpha-1} \diff s \\
&\le  \textstyle
\alpha \sum_{j: x_j y_j >0}|x_j^{1/\alpha} - y_j^{1/\alpha}|   \int_0^1 \{s x_j^{1/\alpha} + (1-s)y_j^{1/\alpha}\}^{\alpha-1} \diff s   \\
&=  \textstyle \sum_{j=1}^d |x_j - y_j|,
\end{align*}
where we used Lipschitz-continuity of $L$, $L(\bm x) \ge \max\{x_1, \dots, x_d\}$ and $\alpha \le 1$.}

As a consequence, we further find that
\begin{align*}
\frac{1}{B_{\POT}(t^{1/\alpha})} \Big| L_0^\alpha\Big( \frac{\lambda_\POT(t/\bm x) }{\lambda_\POT(t) } \Big) 
- L_0^\alpha(\bm x^{1/\alpha}) \Big| 
\le 
\sum_{j: x_j>0} \frac{1}{B_{\POT}(t^{1/\alpha})}  \Big| \Big( \frac{\lambda_\POT(t/x_j) }{\lambda_\POT(t) } \Big)^{\alpha} - (x_j^{1/\alpha})^\alpha \Big|
\end{align*}
and each summand on the right-hand side can be written as 
\begin{multline*}
\frac{\alpha}{B_{\POT}(t^{1/\alpha})}  \Big| \frac{\lambda_\POT(t/x_j) }{\lambda_\POT(t) } - x_j^{1/\alpha} \Big| \times \int_0^1 \Big\{ s  \frac{\lambda_\POT(t/x_j) }{\lambda_\POT(t) } + (1-s) x_j^{1/\alpha}\Big\}^{\alpha-1} \diff s \\
\le 
\frac{\alpha}{B_{\POT}(t^{1/\alpha})}  \Big| \frac{\lambda_\POT(t/x_j) }{\lambda_\POT(t) } - x_j^{1/\alpha} \Big| \times x_j^{1-1/\alpha}  \times  \int_0^1 (1-s)^{\alpha-1} \diff s
\end{multline*}
where we have used that $\lambda_\POT(t/x_j) / \lambda_\POT(t) \ge 0$ and $\alpha\le 1$. By \eqref{eq:lambdalim} below, the right-hand side of the previous display can be written as
\[
O(1) \times \alpha x_j \times \int_0^1 (1-s)^{\alpha-1} \diff s
\]
where the $O(1)$ is uniform for $x_j \in (0,\delta]$. Hence, for sufficiently large $t$, this expression can be made smaller than $\eps/2$ uniformly in $x_j\in(0,\delta]^d$ by decreasing $\delta>0$ if necessary.  
As a consequence, it remains to show that \eqref{eq:con1} holds uniformly in $\bm x \in [0,T]^d \setminus [0,\delta]^d$. For such $\bm x$, using the notation from Lemma~\ref{lem:inv_short} below, 
\begin{align*}
&\,
L_0^\alpha\Big( \frac{\lambda_\POT(t/\bm x) }{\lambda_\POT(t) } \Big)  \\
\stackrel{(a)}{=}&\, 
L_0^\alpha\Big( 
		\bm x^{1/\alpha} 
		+ 
		\big\{ h_{\lambda, \POT}( \bm x^{-1}) + o(1) \big\}
		B_{\POT}(t^{1/\alpha}) 
		\Big)   \\
\stackrel{(b)}{=} &\,
L_0^\alpha(\bm x^{1/\alpha}) 
+ 
\alpha L_0^{\alpha-1}(\bm x^{1/\alpha}+o(1))  
\sum_{j=1}^d  \dot L_{0,j}(\bm x^{1/\alpha}+o(1)) 
		\big\{ h_{\lambda, \POT}( x_j^{-1}) + o(1) \big\}
		B_{\POT}(t^{1/\alpha}) \\
\stackrel{(c)}{=} &\, L_0^\alpha(\bm x^{1/\alpha})  
+ 
\{ S_{\POT,1}(\bm x) + o(1)\} B_{\POT}(t^{1/\alpha}),
\end{align*}
where the $o(1)$ terms in each line are uniform in $\bm x \in [0,T]^d \setminus [0,\delta]^d$. Explanations: (a) follows from  \eqref{eq:lambdalim} below, (b) from a Taylor expansion and (c) from uniform continuity of $\dot L_{0,j}$ on $[0,T]^d \setminus [0,\delta]^d$.

It thus follows from \eqref{eq:decall1} and \eqref{eq:con1} that \ref{cond:sopot} is met with
\[
S_\POT(\bm x) =  
C_\POT^{\rho'/\alpha}  h_{\kappa,\POT} \big(1/{L_0 \big( \bm x^{1/\alpha}  \big)} \big)
+
S_{\POT,1}(\bm x)\footnote{Note that, by Lemma~\ref{lem:reg}, the function $S_\POT$ must be homogeneous of order $1-\rho_\POT'/\alpha$. This can also be verified by some lengthy calculations, using the fact that $h_{\kappa, \POT}(sx) = s^{\rho_\POT'-\alpha}h_{\kappa, \POT}(x)+x^{-\alpha}h_{\kappa, \POT}(s)$ and $L_0(\bm x) = \sum_{j=1}^d x_j \dot{L}_{0,j}(\bm x) $; the latter follows since $L_0$ is homogenous of order $1$.}.
\]

 Next, let $\M=\BM$.  Note that is sufficient to show that the convergence in \eqref{eq:sobm2} is uniform on $[\delta, 1]^d$. For $\bm u\in [\delta,1]^d$, let $\bm x=-\log \bm u \in [0,|\log \delta|]^d$. Then, similarly as before,
\begin{align*}
- r \log C(\bm u^{1/r}) 
&=
	\kappa_\BM\bigg(\frac{1}{\lambda_\BM(r) }
		\frac{1}{L_0\big( \frac{\lambda_\BM(r/\bm x) }{\lambda_\BM(r) } \big) } \bigg) 
	\Big/
	\kappa_\BM\Big( \frac{1}{ \lambda_\BM(r) } \Big)
\\
&= 
L_0^\alpha\Big( \frac{\lambda_\BM(r/\bm x) }{\lambda_\BM(r) } \Big) 
	+
	\bigg\{ h_{\kappa,\BM}\bigg(  
		\frac{1}{L_0\big(\bm x^{1/\alpha} \big) } 
		\bigg) +o(1) \bigg\} C_\BM^{\rho'/\alpha}B_{\BM}(r^{1/\alpha})
\end{align*}
for $r\to\infty$, where all $o$-terms are uniform in $\bm u\in[\delta, 1]^d$, and where $1/0$ is interpreted as $\infty$ and converging functions on $(0,\infty)$ are naturally extended to $(0,\infty]$. The proof is now the same as for $m=\POT$.
\end{proof}

\begin{lemma}
\label{lem:inv_short}
Let $\M \in \{\POT, \BM\}$. Suppose that $\kappa_\M$ satisfies \eqref{eq:sophi1} for some $B_{\M}\in \RV_{\rho_\M'}$ with $\rho_\M'<0$ and denote by $c_\M \neq 0$ the constant in~\eqref{eq:hkappa}. Then the convergence in \eqref{eq:sophi1} is uniform on sets of the form $[\eps,\infty)$ for $\eps > 0$. Moreover, there exists a constant $C_\M>0$ such that, for $t\to\infty$,
\begin{align} \label{eq:lambdarepr}
\lambda_\M(t) =  (C_\M t)^{-{1/\alpha}}\Big\{1 - \frac{1}{\alpha}\frac{c_\M}{\rho_\M'} C_\M^{\rho_\M'/\alpha}  B_{\M}(t^{1/\alpha}) (1+o(1))\Big\}.
\end{align}
Finally,
\begin{align} \label{eq:lambdalim}
\lim_{t \to \infty} \frac{1}{B_{\M}(t^{1/\alpha})}  \Big\{ \frac{\lambda_\M(xt)}{\lambda_\M(t)} - x^{-1/\alpha} \Big\}  
=  
-c_\M \alpha^{-2} C_\M^{\rho'/\alpha} x^{-1/\alpha} \frac{x^{\rho_\M'/\alpha}-1}{\rho_\M'/\alpha} 
=: 
h_{\lambda,\M}(x)
\end{align}
with the latter convergence being uniform on sets of the form $[\eps,\infty)$, $\eps > 0$.
\end{lemma}

\begin{proof}
The proof of uniformity in \eqref{eq:sophi1} is similar to the proof of uniformity in \eqref{eq:lambdalim} and is omitted for the sake of brevity. Since the proof is the same for $\M=\POT$ and $\M=\BM$, we occasionally omit the index $\M$. Recall from the discussion following~\eqref{eq:sophi1} that the functions $f(t) = t^\alpha \kappa(t)$ and $a(t)=t^\alpha \kappa(t) B(t)$ satisfy
\[
\frac{f(tx)-f(t)}{a(t)} \to c \frac{x^{\rho'}-1}{\rho'}
\]
where by assumption $\rho' < 0$ and $c \neq 0$. Assume without loss of generality that $c>0$, otherwise replace $f$ by $-f$ in the arguments that follow and make corresponding adjustments. By part 2 of Theorem B.2.2 in \cite{DehFer06} the limit $C_\M := C := \lim_{t \to \infty} f(t)$ exists and we have
\[
C - t^\alpha \kappa(t) =  - \frac{c}{\rho'}t^\alpha \kappa(t) B(t)(1+o(1)),
\] 
which implies 
\[
\kappa(t) = C t^{-\alpha}  \big\{ 1 + \frac{c}{\rho'}B(t)(1+o(1))\big\}.
\]
Note that $C \geq 0$ since $\kappa(t) \geq 0$ and that $C \neq 0$ since otherwise Theorem~B.2.2 in \cite{DehFer06} would imply $-1/B(t) = - f(t)/a(t) \to -c/\rho'$, which contradicts $B(t) = o(1)$. Next note that $1/t=\kappa(1/\lambda(t))$ for all sufficiently large $t$. Hence, by the previous display,
\begin{align}\label{eq:helplambda}
t^{-1} = \lambda(t)^{\alpha}C\big\{ 1 + \frac{c}{\rho'}B(1/\lambda(t))(1+o(1))\big\}.
\end{align}
Since $\lambda(t)=o(1)$ (by definition of $\lambda$), we first obtain $t^{-1} = \lambda(t)^{\alpha}C(1+o(1))$, i.e. $\lambda(t) = (t C)^{-1/\alpha}(1+o(1))$. By regular variation of $B$ combined with the Uniform Convergence Theorem (see Theorem B.1.4 in \citealp{DehFer06}) we obtain
\begin{align} \label{eq:B1lambda}
B_\M(1/\lambda_\M(t)) = B_\M(t^{1/\alpha})C_\M^{\rho_\M'/\alpha}(1+o(1)),
\end{align}
{where we added the indices for ease of reference.}
After plugging this into~\eqref{eq:helplambda} and rearranging terms, a Taylor expansion implies that
\begin{align*}
\lambda(t) 
&= 
(C t)^{-1/\alpha}  \big\{ 1 + \frac{c}{\rho'}C^{\rho'/\alpha}B(t^{1/\alpha})(1+o(1) \big\}^{-1/\alpha} 
\\
&=
(C t)^{-1/\alpha} \big\{ 1  - \frac{c}{\alpha \rho'}C^{\rho'/\alpha}B(t^{1/\alpha})(1+o(1)\big\},
\end{align*}
and hence~\eqref{eq:lambdarepr} follows. The fact that~\eqref{eq:lambdalim} holds for any fixed $x$ follows from~\eqref{eq:lambdarepr} by straightforward calculations. 

Hence it remains to prove that the convergence in~\eqref{eq:lambdalim} is actually uniform. Letting $f(t) := -c \alpha^{-2} C^{\rho'/\alpha}t^{1/\alpha}\lambda(t)$ and $a(t) := t^{1/\alpha}\lambda(t)B(t^{1/\alpha})$, the claim can be equivalently formulated as 
\[
\lim_{t\to\infty} \frac{f(tx)-f(t)}{a(t)} = \frac{x^{\rho'/\alpha}-1}{\rho'/\alpha}
\]
uniformly in $x\in[\eps, \infty)$. It follows from the properties of $\lambda$ established earlier that $f(\infty) = \lim_{t\to\infty} f(t) = -c \alpha^{-2} C^{(\rho'-1)/\alpha}$ exists. Let $a_0(t) = |\rho'/\alpha|\{f(\infty) - f(t)\}$ and note that $\lim_{t\to\infty} a_0(t)/a(t) = 1$ by Theorem B.2.2 in \cite{DehFer06}. Moreover, by Theorem B.2.18 in that reference, we can find, for any $\delta>0$, some $t_0>0$ such that, for all $t\ge t_0$,
\[
\sup_{x\ge \eps} \Big| \frac{f(tx)-f(t)}{a_0(t)} - \frac{x^{\rho'/\alpha}-1}{\rho'/\alpha} \Big| \le  \delta \sup_{x\ge \eps}  \max\{x^{\rho'/(2\alpha)}, x^{3\rho'/(2\alpha)}\} = \delta \eps^{3\rho'/(2\alpha)}.
\]
In other words,
\[
\lim_{t\to\infty} \sup_{x\ge \eps} \Big| \frac{f(tx)-f(t)}{a_0(t)} - \frac{x^{\rho'/\alpha}-1}{\rho'/\alpha} \Big| = 0.
\]
The decomposition
\[
\frac{f(tx)-f(t)}{a(t)}  - \frac{x^{\rho'/\alpha}-1}{\rho'/\alpha} 
= 
\frac{a_0(t)}{a(t)}  \Big\{ \frac{f(tx)-f(t)}{a_0(t)}  - \frac{x^{\rho'/\alpha}-1}{\rho'/\alpha}  \Big\} + 
\Big\{ \frac{a_0(t)}{a(t)} - 1 \Big\}  \frac{x^{\rho'/\alpha}-1}{\rho'/\alpha}
\]
then implies the assertion.  
\end{proof}

\section{Consequences for estimation precision} \label{sec:est}

In this section, we illustrate the consequences of the results in the previous sections for estimating multivariate extremal dependence. Suppose $\bm X_1, \dots, \bm X_n$ is a finite sample from a distribution with continuous marginal cdfs and copula $C$ such that \eqref{eq:dombm} or, equivalently, \eqref{eq:dompot} is met. For the sake of a clear exposition, we will concentrate on the bivariate case and on estimating the Pickands dependence function $A(t)=L(1-t,t)$ based on one estimation method from the BM and POT approach, respectively. 

For the POT approach, we concentrate on the empirical stable tail dependence function, evaluated at the point $(1-t,t)$, defined as
\[
\hat A_\POT(t) =  \frac1k \sum_{i=1}^n \ind(\hat F_{n,1}(X_{i,1}) > 1 - (1-t) \tfrac{k}{n} \text{ or } \hat F_{n,2}(X_{i,2}) > 1- t \tfrac{k}{n}),
\]
where $\hat F_{n,j}$ denotes the $j$th marginal empirical cdf, for $j=1,2$. Let
\[
A_n(t) = \frac{1-C(1-(1-t)k/n, 1-t k/n )}{k/n} = \frac{\Prob(U_1 > 1-(1-t)k/n \text{ or } U_2 >1-t k/n )}{k/n}
\]
denote the pre-asymptotic version of $A(t)$, where $(U_1, U_2) \sim C$. The asymptotic analysis of $\hat A_\POT$ is based on the bias-variance type decomposition 
\[
\sqrt k \{ \hat A_\POT(t) - A(t) \} 
=
\sqrt k \{ \hat A_\POT(t) - A_n(t)\}  + \sqrt k \{ A_n(t) - A(t)\}.
\]
First, it follows from simple adaptations of the results in \cite{Hua92} that
\[
\sqrt k \{ \hat A_\POT(t) - A_n(t)\}
\]
is asymptotically centered normal (in fact, even functional weak convergence holds), provided that $k\to\infty$ and $k=o(n)$ as $n\to\infty$ and that some mild regularity conditions on $L$ are met. Moreover, if \ref{cond:sopot} is met, the dominating bias term satisfies
\[
\sqrt k \{ A_n(t) - A(t)\} 
=
\sqrt k \alpha_\POT(n/k) \frac{A_n(t) - A(t)}{\alpha_p(n/k)} = 
\sqrt k \alpha_\POT(n/k) \{ S_\POT(1-t, t) + o(1)\}
\]
as $n\to\infty$. Hence, in the typical case $\alpha_\POT(t) = c t^{\rho_\POT}$ for some $\rho_\POT<0$, choosing $k$ proportional to $n^{-2\rho_\POT/(1-2\rho_\POT)}$ leads to the best convergence rate for this estimator, with the resulting rate being of order $n^{\rho_\POT/(1-2\rho_\POT)}$. Under additional assumptions, \cite{DreHua98} proved that this is the minimax-optimal convergence rate.

Next, consider the block maxima method. 
For some block size $r\in\{1,\dots, n\}$, decompose the data into $k= \lfloor{n/r}\rfloor$ disjoint blocks of size $r$, that is, let the $i$th block maxima in coordinate $j$ be defined as
\[
M_{r,i,j} = \max\{  X_{t,j}: t=(i-1)r+1, \dots, ir\}.
\]
By \eqref{eq:dombm}, the copula of the i.i.d.\ sample $(M_{r,i,1}, M_{r,i,2}), i=1, \dots, k,$ is approximately given by $C_\infty$, whence $A$ can be estimated by any method of choice, like the Pickands or CFG-estimator. For simplicity, we concentrate on the madogram whose asymptotic behavior can be immediately deduced from the results in \cite{BucSeg14}. Let $\hat U_{i,j} = \hat G_{k,j}(M_{r,i,j})$, where $\hat G_{k,j}$ denotes the $j$th marginal empirical cdf of the sample of block maxima. The madogram-based estimator for $A$ is defined as
\[
\hat A_\BM(t) = \frac{\hat \nu(t)}{1-\hat \nu(t)}, \qquad \hat \nu(t) = \frac{1}k \sum_{i=1}^k \max\{ \hat U_{i,1}^{1/(1-t)}, \hat U_{i,2}^{1/t} \},
\]
which is motivated by the fact that 
\[
A(t) = \frac{\nu(t)}{ 1- \nu(t)}, \qquad \nu(t) = \Exp[\max(U_1^{1/(1-t)}, U_2^{1/t})].
\]
Under Condition~\ref{cond:sobm} it can be shown that the best rate of convergence which can be attained by this estimator is $n^{\rho_\BM/(1-2\rho_\BM)}$. We sketch a proof: first, by some simple calculations, one may write
\begin{align} \label{eq:sc}
\nu(t) = 1- \int_0^1 C_\infty(y^{1-t}, y^t) \diff y, \qquad 
\hat \nu(t) = 1- \int_0^1 \hat C_\BM(y^{1-t}, y^t) \diff y,
\end{align}
where $\hat C_\BM$ denotes the empirical cdf of the sample $(\hat U_{1,1},\hat U_{1,2}), \dots, (\hat U_{k,1},\hat U_{k,2})$, i.e., the empirical copula of the sample of block maxima. Second, by Corollary~3.6 in \cite{BucSeg14}, the process
\[
\sqrt k \{\hat C_\BM(\bm u) - C_r(\bm u) \}
\]
converges weakly to a centred Gaussian process, provided $r\to\infty, r/n\to0$ and some regularity conditions on $C_\infty$ are met. Under Condition~\ref{cond:sobm}, the decomposition
\begin{align*}
\sqrt k \{\hat C_\BM(\bm u) - C_\infty(\bm u) \} 
&= 
\sqrt k \{\hat C_\BM(\bm u) - C_r(\bm u) \}  + \sqrt k \{ C_r(\bm u) - C_\infty(\bm u) \} \\
&=
\sqrt k  \{\hat C_\BM(\bm u) - C_r(\bm u) \}  + \sqrt k \alpha_\BM(r) \{ S_\BM(\bm u) + o(1)\}
\end{align*}
then shows that we obtain a proper (possibly non-centred) limit process if we choose $r$ such that  $\sqrt k \alpha_\BM(r) $ is converging. In the typical case $\alpha_\BM(t) = c t^{\rho_\BM}$, choosing $r$ proportional to $n^{1/(1-2\rho_\BM)}$ leads to the optimal convergence rate $n^{\rho_\BM/(1-2\rho_\BM)}$. By standard arguments based on the continuous mapping theorem and \eqref{eq:sc}, the convergence rate of the empirical copula easily transfers to $\hat A_\BM$.

We illustrate the results in the preceding two paragraphs with three example models from the Archimax family. The parameter $L_0$ is chosen as the stable tail dependence function from the Gumbel copula (also known as the symmetric logistic model), that is,
\[
L_0(x_1, x_2) = (x_1^\theta + x_2^\theta)^{1/\theta}.
\]
Throughout, we fix $\theta=\log(2) / \log(3/2)$ such that $A_0(1/2)=A_0(1/2,1/2)=3/4$ and $\lambda=1/2$. We consider the following three Archimedean generators
\begin{align*} 
\psi_1(x) & = (1-x+x^2/4) \bm 1(x \in [0,1/2]) + (15/16 -3x/4) \bm 1(x \in  (1/2, 5/4]), \nonumber \\
\psi_2(x) & = (1-x+x^2/2)\bm 1(x \in [0,1/2]) + (7/8 -x/2) \bm 1(x \in  (1/2, 7/4]), \\
\psi_3(x) & = (1-x+x^3/6)\bm 1(x \in [0,1/2]) + (23/24 -7x/8) \bm 1(x \in  (1/2, 23/21]), \nonumber
\end{align*}
which are continuously differentiable on their respective supports. The second order expansions of the associated functions $\kappa_\BM$ and $\kappa_\POT$ at $\infty$ are given by

\begin{center}
\renewcommand{\arraystretch}{1.2}
\begin{tabular}{c|cc|cc}
  \hline
  \hline
  Generator &  $\rho_\POT'$ &  $\rho_\BM'$ & Expansion $\kappa_\POT$ & Expansion $\kappa_\BM$   \\ \hline
  $\psi_1$ & $-1$ & $-1$ & $\frac1x-\frac1{4x^2}$  & $\frac1x+\frac{1}{4x^2}+O(x^{-3})$ \\
     $\psi_2$ & $-1$ & $-2$  & $\frac1x-\frac1{2x^2}$  & $\frac1x - \frac{1}{6x^3}+O(x^{-4})$  \\
     $\psi_3$ & $-2$ & $-1$  & $\frac1x-\frac1{6x^3}$  & $\frac1x + \frac{1}{2x^2}+O(x^{-3})$  \\
  \hline \hline
\end{tabular}
\end{center}

\noindent
Note the symmetry between the expansions of $\kappa_\BM(\psi_2)$ and $\kappa_\POT(\psi_3)$. By  Proposition~\ref{prop:archso}, the corresponding second order parameters $\rho_\M=\rho_\M(\psi_j), \M \in \{\BM, \POT\},$ of the associated archimax copulas are given by $\rho=\rho'/\alpha=\rho'$ and thus
\[
\rho_{\BM}(\psi_j) = \begin{cases} -1 &, j=1 \\ -2 &, j=2 \\ -1 &, j=3 \end{cases} 
\qquad \qquad
\rho_{\POT}(\psi_j) = \begin{cases} -1 &, j=1 \\ -1 &, j=2 \\ -2 &, j=3. \end{cases}
\]
Hence, if either the threshold $k$ (POT-estimator) or the block size $r$ (BM-estimator) is chosen at the best attainable rate of convergence as indicated at the beginning of this section, one would expect that estimators behave similarly for $j=1$ (convergence rate $n^{-1/3}$), that the BM-estimator outperforms the POT-estimator for $j=2$ (convergence rate $n^{-2/5}$ vs.\ $n^{-1/3}$), and vice versa for $j=3$. 

Let $\Gamma_\BM=\{1,2,3,\dots, 30\}$ denote a set of block sizes, and let $\Gamma_\POT=\Gamma_\POT(n) = \{k=\lfloor{pn}\rfloor: p \in \{0.01, 0.02, \dots, 0.39, 0.4\}\}$ denote a set of thresholds. In Table~\ref{tab:releff}, we state the relative efficiency
\[
\text{RE} = \frac{\min_{r \in \Gamma_\BM} \MSE(\hat A_{\BM,r}(1/2) ) }{\min_{k \in \Gamma_\POT}  \MSE(\hat A_{\POT,k}(1/2) ) } 
\]
of the \textit{best} (optimal choice of $r$) BM-estimator $\hat A_\BM(1/2)=\hat A_{\BM,r}(1/2)$ to the \textit{best} (optimal choice of $k$) POT-estimator $\hat A_\POT(1/2)=A_{\POT,k}(1/2)$, considered as estimators for $A(1/2)=3/4$, for four different sample sizes $n=1000, 2000, 5000, 10000$. The values are calculated based on $3000$ Monte Carlo repetitions. Simulated samples from the Archimax copulas are generated by the algorithm described in Section 5.2 in \cite{ChaFouGen14}. The results perfectly match the expected behavior: for model $\psi_1$, the relative efficiencies are close to 1 (in fact, they are all slightly above 1), while they are decreasing for  $\psi_2$ and increasing for $\psi_3$.

\begin{table}[!htbp]
\centering
\renewcommand{\arraystretch}{1.0}
\begin{tabular}{r|rrr}
  \hline
  \hline
  Sample size $n$ & $\psi_1$ & $\psi_2$  & $\psi_3$ \\ \hline
1000 & 1.214 & 0.233 & 3.253 \\ 
  2000 & 1.127 & 0.198 & 3.751 \\ 
  5000 & 1.141 & 0.147 & 4.104 \\ 
  10000 & 1.088 & 0.123 & 4.762 \\  
  \hline \hline
\end{tabular}
\medskip
\caption{Relative efficiencies for estimating $A_0(1/2,1/2)$  based on Monte Carlo Simulation. Values below 1 indicate that the BM-estimator is more efficient.}\label{tab:releff}
\end{table}

As a further illustration, Figure~\ref{fig:MSE} depicts variance, squared bias and MSE as a function of the block size~$r$ (BM-estimator) or the threshold parameter $k$ (POT-estimator) for fixed sample size $n=5000$; again based on 3000 Monte Carlo replications. 
The following observation can be made estimator-wise: the variance curves behave similarly for all models, while the squared bias curve is much smaller for the respective model with $\rho_{\M}=-2$ than for the other two models.

\begin{figure}[!htbp]
\begin{center}
\centerline{\includegraphics[width=0.975\textwidth]{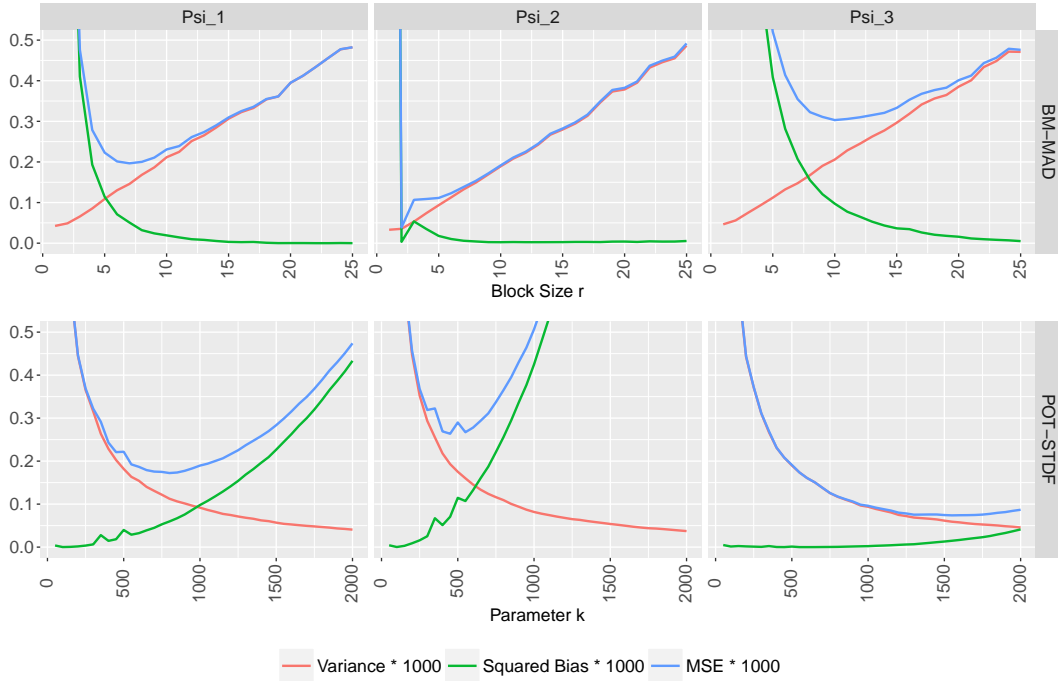}}
\vspace{-.5cm}
\caption{\label{fig:MSE} Variance, Squared Bias and MSE for estimating $A_0(1/2,1/2)$ based on Monte Carlo Simulation (3000 repetitions, Sample Size $n=5000$). }
\end{center}
\end{figure}

\section*{Acknowledgments}

Axel Bücher gratefully acknowledges support by the Collaborative Research Center ``Statistical modeling of nonlinear dynamic processes'' (SFB 823) of the German Research Foundation, Project A7. Stanislav Volgushev gratefully acknowledges support by a discovery grant from NSERC of Canada.

\newpage

\bibliographystyle{chicago}
\bibliography{biblio}
\end{document}